\documentclass[reqno]{amsart}
\usepackage{amsmath}
\usepackage{amssymb, hyperref}
\usepackage{graphicx,color}
\usepackage{slashed}
\renewcommand{\Im}{\operatorname{Im}}
\usepackage{mathtools}
\newcommand{\R}{\mathbb{R}}

\usepackage{parskip}
\usepackage{xcolor}

\numberwithin{equation}{section}
\newtheorem{theorem}{Theorem}[section]
\newtheorem{proposition}[theorem]{Proposition}
\newtheorem{lemma}[theorem]{Lemma}
\newtheorem{corollary}[theorem]{Corollary}
\newtheorem{remark}[theorem]{Remark}
\newtheorem{definition}[theorem]{Definition}
\mathtoolsset{showonlyrefs}

\begin{document}

\title[On the gH equation with potential]{Scattering for the generalized Hartree equation with a potential}

    \author[Carlos M. Guzmán]{CARLOS M. GUZM\'AN}  
	\address{CARLOS M. GUZMAN  \hfill\break
	Universidade Federal Fluminense, Instituto de Matemática e Estatistica, Brazil.}
	\email{carlos.guz.j@gmail.com}
    
    \author[Cristian Loli]{CRISTIAN LOLI}  
	\address{CRISTIAN LOLI  \hfill\break
	Universidade Federal Fluminense, Instituto de Matemática e Estatistica, Brazil.}
	\email{cristianloli@id.uff.br}

    \author[Luis P. Yapu]{LUIS P. YAPU} 
	\address{LUIS P. YAPU \hfill\break
	Universidade Federal Fluminense, Instituto de Matemática e Estatistica, Brazil.}
	\email{luis.yapu@gmail.com}

\begin{abstract}
We consider the focusing generalized Hartree equation in $H^1(\R^3)$ with a potential,
\begin{equation*}
iu_t + \Delta u - V(x)u + (I_\gamma \ast |u|^p )|u|^{p-2} u=0,
\end{equation*}
where  $I_\gamma = \frac{1}{|x|^{3-\gamma}}$, $p \geq 2$ and $\gamma < 3$. In this paper, we prove scattering for the generalized Hartree equation with a potential in the intercritical case assuming radial initial data. The novelty of our approach lies in the use of a general mass-potential condition, incorporating the potential
V, which extends the standard mass-energy framework. To this end, we employ a simplified method inspired by Dodson and Murphy \cite{Dod-Mur}, based on Tao's scattering criteria and Morawetz estimates. This approach provides a more straightforward proof of scattering compared to the traditional concentration-compactness/rigidity method of Kenig and Merle \cite{KENIG}.


\
  
\end{abstract}


\maketitle  

\section{\bf INTRODUCTION}\label{sec1}
\indent In this paper, we investigate the Cauchy problem for the focusing generalized Hartree equation with a potential (GHPV for short). Let $u = u(x, t)$ be a complex-valued function. We consider 
\begin{equation} \label{GHP}
\begin{cases}
    iu_t + \Delta u - V(x)u + (I_\gamma * |u|^p )|u|^{p-2} u=0,\text{  } in\text{  } \mathbb{R}^3\times(0,T)\\
    u(x,0)=u_0(x)  \in H^1(\R^3),
\end{cases}
\end{equation}
where  $I_\gamma = \dfrac{1}{|x|^{3-\gamma}}$, $p \geq 2$ and $0<\gamma < 3$. The potential $V:\R^3 \rightarrow \R$ satisfies the following conditions:
\begin{equation}\label{poten}
   \quad V \in K_0\cap L^{\frac{3}{2}}, 
\end{equation}
\begin{equation}\label{poten2}
||V_{-}||_{K_0}< 4\pi, \quad \text{where} \quad V_{-}=\min\{V,0\}.
\end{equation}
Here, $K_0$ represents the closure of the set of functions with compact support concerning the Kato norm:
\begin{equation}\label{Kato}
    \|V\|_{K_0}=\underset{x \in \R^3}{\sup}\int_{\R^3} \dfrac{|V(y)|}{|x-y|}dy.
\end{equation}

\ Under the assumptions of \eqref{poten}-\eqref{poten2}, it is established that the operator \( H := -\Delta + V \) possesses no eigenvalues, and the Sobolev norms \(\|\Lambda f\|_{L^2}\) and \(\|\nabla f\|_{L^2}\) are equivalent, where
\[
\|\Lambda f\|_{L^2}^2 := \int_{\mathbb{R}^3} |\nabla f|^2 \, dx + \int_{\mathbb{R}^3} V |f|^2 \, dx.
\]


\ The solutions to equation \eqref{GHP} conserve both mass and energy throughout their lifespan, given by the expressions:
\begin{equation*}
    M(u)=\int_{\R^3} |u(x,t)|^2 dx = M(u_0),
\end{equation*}
\begin{equation}
\label{eq:energy}
    E(u)=\dfrac{1}{2}||\nabla u||^2_{L^2}-\dfrac{1}{2p}\int_{\R^3}  (I_\gamma \ast |u|^p )|u|^p dx + \dfrac{1}{2} \int_{\R^3} V(x)|u|^2 dx = E(u_0).
\end{equation}
The energy $E_0$ without the potential $V$, is defined by
\begin{equation*}
E_0(u)=\dfrac{1}{2}||\nabla u||^2_{L^2}-\dfrac{1}{2p}\int_{\R^3}  (I_\gamma \ast |u|^p )|u|^p dx ,
\end{equation*} 
and we denote the potential energy by
\begin{equation}\label{potential}
 P(u(t))=\int_{\R^3}  (I_\gamma \ast |u|^p )|u|^p dx.
\end{equation}

\ The main purpose of this paper is to study the energy scattering for \eqref{GHP} in the intercritical case, that is, $0 < s_c < 1$. The critical exponent $s_c$, which defines a scale-invariant Sobolev norm in $\Dot{H}^{s_c}$, is given by
\begin{equation*}
s_c=\dfrac{3}{2}-\dfrac{\gamma+2}{2(p-1)}.
\end{equation*}
\begin{definition} We say that a solution to \eqref{GHP} \emph{scatters} in $H^1$ if there exist $u_\pm\in H^1$ such that
\[
\lim_{t\to\pm\infty}\|u(t) - e^{-itH}u_\pm\|_{H^1}=0,
\]
where $e^{-itH}$ is the free Schr\"odinger propagator. 
\end{definition}
\ To characterize the threshold for scattering, we require a ground state. The equation \eqref{GHP} admits wave solutions of the form $u(x, t) = e^{it} Q(x)$, where $Q$ solves the nonlinear elliptic equation
\begin{equation}\label{elip}
    -Q+\Delta Q+(I_\gamma*|Q|^p)|Q|^{p-2}Q=0.
\end{equation}
The first result of the existence and uniqueness for $p=2$ and $\gamma=2$ was proved by Lieb \cite{Lieb-Choquard-77}. This solution is smooth, radial, decreasing in the radial coordinate and exponentially decaying at infinity. 
 In the general case, the existence of positive solutions for $\frac{N + \gamma}{N} < p < \frac{N + \gamma}{N - 2}$ was shown by Moroz and van Schaftingen \cite{MOROZ-VSchaf-2013}. While the uniqueness of these solutions remains an open question, we utilize the minimizer of the Gagliardo-Nirenberg inequality and its computed value from \cite{ARORA-ROUDENKO-22} for our purposes.

\ Before stating our results, we review previous findings on the model \eqref{GHP}. The well-posedness of the Hartree equation without a potential (\(V=0\) and \(p=2\) in \eqref{GHP}) has been extensively studied, starting with the works of Cazenave \cite{CAZENAVEBOOK} and Ginibre \& Velo \cite{GV}. The local and global well-posedness in \(H^1\) for the generalized Hartree equation without potential was further investigated in \cite{ARORA-ROUDENKO-22}. Scattering results have generally been derived using the compactness and rigidity framework of Kenig and Merle \cite{KENIG}. A simplified proof of scattering for the cubic NLS equation, employing a scattering criterion introduced by Tao \cite{Tao}, was provided by Dodson and Murphy \cite{Dod-Mur}. For the generalized Hartree equation with \(N \geq 3\) and without potential, scattering was proved by Arora and Roudenko following the Kenig-Merle roadmap \cite{ARORA-ROUDENKO-22}, and by Arora \cite{ARORA-GH-19} using a simplified proof similar to the method of \cite{Dod-Mur}. Here, we study the generalized Hartree equation with a potential, as given by equation \eqref{GHP}, in the inter-critical case, utilizing ideas from \cite{Dod-Mur} (see also \cite{Gao_Wang-20-hartree}). Initially, we establish a global existence result for small data in $H^1$ across the entire inter-critical regime.

\begin{proposition} 
Let $p \geq 2$ satisfy $\frac{5+\gamma}{3} < p < 3+ \gamma$ with $0 < \gamma < 3$. Let $V$ satisfy \eqref{poten} and
\eqref{poten2}. Suppose that $\Vert u_0 \Vert_{H^1} \leq E$ and take $\bar r=\frac{12(p-1)}{3+2\gamma-2\epsilon}$, $\bar a=\frac{8(p-1)}{1+2\epsilon}$. Then there exists $\delta = \delta(E)$ such that if
$$ \Vert e^{-itH} u_0 \Vert_{L_t^{\bar a} L_x^{\bar r}} < \delta,$$
then there exist a unique solution $u$ to \eqref{GHP} with initial condition $u_0 \in H^1(\R^3)$ which is globally defined on $[0, \infty)$ such that
$$\Vert u \Vert_{L_t^{\bar a} L_x^{\bar r}} <  2\delta\quad \textnormal{and}\quad \Vert u \Vert_{S(L^2)} + \Vert \Lambda u \Vert_{S(L^2)} \leq 2 C \Vert u_0 \Vert_{H^1}.$$
\end{proposition}
\begin{remark}
It is important to note that, due to the presence of the potential \( V \), we must use admissible pairs \( (q, r) \) with \( r < 3 \), rather than \( r \leq 6 \) as in the case without a potential. In this work, we use the pair \( (q, r) = (4^+, 3^-) \); see Section \ref{sec2}. Additionally, the pair \( (\bar{a}, \bar{r}) \) used here represents the scattering norm and is not \( H^{s_c} \)-admissible because \( \bar{r} \) does not satisfy \( \bar{r} < 6 \). Consequently, proving the results with a potential differs from the model without a potential. If we do not consider the appropriate pairs and carefully estimate, we would need to restrict the range of the parameter \( p \). This is one of the main contributions of this work.
\end{remark}

\ We now present the main result of this paper. It is stated as follows.
\begin{theorem} (Scattering)
\label{theo:scattering}
Consider the generalized Hartree equation \eqref{GHP} with $p \geq 2$, $\gamma < 3$ and $\frac{5+\gamma}{3} < p < \gamma + 3$, with potential  $V:\R^3 \rightarrow \R$ satisfying \eqref{poten}, $V \geq 0$, $x \cdot \nabla V \leq 0$, and $x \cdot \nabla V \in L^r$ with $r \in [\frac{3}{2},\infty)$. Let $u_0 \in H^1(\R^3)$ be radial and $u(t)$ be the corresponding solution of \eqref{GHP} such that 
\begin{equation}
\label{eq:only_one_hypo}
    \underset{t \in [0,T^*]}{sup} P(u(t))M(u(t))^{\sigma_c} < P(Q)M(Q)^{\sigma_c}, \quad \text{where} \quad \sigma_c=\frac{1-s_c}{s_c},
\end{equation}
then\footnote{Note that we do not require the condition \eqref{poten2} for $V$, as we assume $V\geq 0$.} $u$ is global solution $(T^*=\infty)$ and scatters in $H^1$.
\end{theorem}
\begin{remark}
Note that, by mass and energy conservation, the hypothesis \eqref{eq:only_one_hypo} implies that the solution of \eqref{GHP} is global and the norm $ \|u\|_{L_t^\infty H_x^1}$ is bounded.
\end{remark}

\ The hypothesis \eqref{eq:only_one_hypo} is more general than the standard mass-energy threshold conditions \eqref{eq:cond1_ME}-\eqref{eq:cond2} typically used to prove scattering. This is formalized in the following result. 

\begin{corollary}
\label{two_conds_one_cond-intro}
Let $u$ be a solution of \eqref{GHP} with maximal lifespan. If 
\begin{equation}\label{eq:cond1_ME}
    M(u_0)^{\sigma_c} E(u_0) < M(Q)^{\sigma_c} E_0(Q)
\end{equation}
and
\begin{equation}\label{eq:cond2}
\|u_0\|_2^{\sigma_c} \|\Lambda u_0\|_2 < \|Q\|_2^{\sigma_c} \|\nabla Q\|_2.
\end{equation}
Then
\begin{equation}
\label{eq:sup_u_Lambda_u}    
\sup_{t \in [0,\infty)} \|u(t)\|_2^{\sigma_c} \|\Lambda u(t)\|_2 < \|Q\|_2^{\sigma_c} \|\nabla Q\|_2.
\end{equation}
That is, $u$ is a global. Moreover, $u$ scatters in $H^1$.%
\end{corollary}

\ The proof of Theorem \ref{theo:scattering} consists of
two main steps: First, we establish the scattering criterion for the Hartree model with potential. This criterion, initially proven by Tao \cite{Tao} for the radial 3D cubic NLS equation, is adapted to our context. Specifically, we show scattering under the conditions specified in Theorem \ref{theo:scattering} and the assumption
$$
\liminf_{t \rightarrow +\infty}\int_{B(0,R)}|u(x,t)|^2 \, dx \leq \varepsilon^2.
$$
The second step is to verify that this inequality holds, which is ensured by satisfying the assumption stated in \eqref{eq:only_one_hypo}. To do that, we employ Virial-Morawetz-type estimates.

\ The paper is organized as follows. In Section 2 we study the results of global well-posedness and scattering for small data. In Section 3 we prove the Scattering criterion following the method of Dodson and Murphy \cite{Dod-Mur} based on arguments of Tao \cite{Tao}. Section 4 is devoted to the variational analysis using the hypothesis \eqref{eq:only_one_hypo}, and we prove Corollary \ref{two_conds_one_cond-intro}. Finally, in Section 5 we prove Theorem \ref{theo:scattering} using Virial/Morawetz estimates.

\section{\bf NOTATION AND PRELIMINARIES}\label{sec2}

\ Let us start this section by introducing the notation used throughout the paper. We write $a \lesssim b$ to denote $a \leq cb$ for some constant $c > 0$. If $a \lesssim b \lesssim a$, we write $a \sim b$.

We make use of the standard Lebesgue spaces $L^p$, the mixed Lebesgue spaces $L_t^q L_x^r$, as well as the homogeneous and inhomogeneous Sobolev spaces $\dot{H}^{s,r}$ and $H^{s,r}$. When $r = 2$, we write $\dot{H}^{s,2} = \dot{H}^s$ and $H^{s,2} = H^s$. If necessary, we use subscripts to specify which variable we are concerned with. We use $'$ to denote the Hölder dual of an index.

Consider the operator $\mathcal{H} := -\Delta+V$ and $\Lambda := (-\Delta+V)^{\frac{1}{2}}$.
Under the conditions \eqref{poten} on $V$, the operator $\mathcal{H}$ has no eigenvalues, $\Lambda$ is well-defined and the operator $e^{-it\mathcal{H}}$ enjoys dispersive estimates \eqref{eq:dispersive_estimate} and Strichartz estimates \eqref{SE1}-\eqref{SE2}, see  Hong \cite{HONG}.

Let us consider the norms,
$$
\|u\|_{\dot W_V^{s,r}} := \|\Lambda^s u\|_{L^r} \qquad \text{and} \qquad  \|u\|_{W_V^{s,r}} := \| \langle \Lambda \rangle^s u\|_{L^r} \sim \|u\|_{L^r} + \|\Lambda^s u\|_{L^r}.
$$

\begin{lemma}[\cite{HONG}]
If $1<r<\frac{3}{s}$, where $0 \leq s \leq 1$, then we have the following equivalence of the norms,
$$
\|u\|_{\dot W_V^{s,r}} \sim \|u\|_{\dot W^{s,r}} \qquad \text{and} \qquad \|u\|_{W_V^{s,r}} \sim \|u\|_{W^{s,r}}.
$$
\end{lemma}

\ In the particular case when $V\geq 0$, we have
$$
\|\Lambda u\|_{L^2}^2 = \int_{\R^3} |\nabla u|^2 dx + \int_{\R^3} V(x)|u|^2 dx \geq \|\nabla u\|_{L^2}^2.
$$

\begin{lemma}(Sobolev inequality)
\label{lema:sobolev_ineq}
Let \( V : \mathbb{R}^3 \to \mathbb{R} \) satisfy \eqref{poten}. Then it holds that
\[
\|f\|_{L^q} \lesssim \|f\|_{\dot{W}^{s, r}_V} \qquad \|f\|_{L^p} \lesssim \|f\|_{W^{s, r}_V(\mathbb{R}^3)},
\]
where \( 1 < p < q < \infty \), \( 1 < p < 3 \), \( 0 \leq s \leq 2 \), and \( \frac{1}{q} = \frac{1}{p} - \frac{s}{3} \).
\end{lemma}


\begin{lemma} (Radial Sobolev)
\label{radial-sobolev}
For a radial function $f \in H^1$ and $\frac{1}{2} \leq s \leq 1$, it follows that
$$ \| |x|^s f \|_{L^\infty} \lesssim \|f\|_{H^1}.
$$    
\end{lemma}

We also have the following version:
\begin{lemma}[\cite{Hamano-Ikeda-20}] (Radial Sobolev in $L^p$).
\label{radial-sobolev-Lp}
For $p\geq 1$ and a radial function $f \in H^1$, it follows that
$$
\|f\|_{L^{p+1}(|x|\geq R)}^{p+1} \lesssim \frac{1}{R^{p-1}} \|f\|_{L^2(|x|\geq R)}^{\frac{p+3}{2}} \|\nabla f\|_{L^2(|x|\geq R)}^{\frac{p-1}{2}},
$$
for any $R>0$.
\end{lemma}

\ Now we recall the Hardy-Littlewood-Sobolev estimate, which allows us to handle the convolution term in the generalized Hartree equation.

\begin{lemma} 
\label{lemma:Hardy}
For $0<\gamma<N$ and $r>1$, there exists a constant $C=C(r,\gamma)$ such that
$$
\left\| \int_{\R^N} \frac{u(y)}{|x-y|^{N-\gamma}} dy\right\|_{L^r(\R^N)} \leq C \|u\|_{L^q(\R^N)},
$$
where $\frac{1}{q} = \frac{1}{r}-\frac{\gamma}{N}$ and $r<\frac{N}{\gamma}$.
\end{lemma}

\subsection{Well-posedness theory}

To discuss the well-posedness theory for \eqref{GHP}, we first recall the Strichartz estimates for the generalized Hartree equation without potential from  \cite{ARORA-ROUDENKO-22} for $N=3$. We say the pair $(q, r)$ is $L^2$-admissible if it satisfies
\begin{equation}
\label{L2_equation}
\frac{2}{q}+\frac{3}{r}=\frac{3}{2},    
\end{equation}
where 
\begin{equation}
\label{L2Admissivel}
 q\geq 2 , \quad 2\leq  r  \leq 6.
\end{equation}
Here, $n^+$ denotes a number (slightly) greater than $n$ such that $\frac{1}{n} = \frac{1}{n^+} + \frac{1}{(n^+)'}$. Analogously, $n^-$ denotes a number (slightly) less than $n$.

Given a real number $s>0$, we also say that the pair $(q, r)$ is $S(\dot H^s)$-admissible if  
\begin{equation}
\label{Hs_equation}
\frac{2}{q}+\frac{3}{r}=\frac{3}{2}-s,    
\end{equation}
with
\begin{equation}
\label{HsAdmissivel}
\left(\frac{2}{1-s}\right)^+ \leq q \leq \infty, \quad
\frac{6}{3-2s} \leq  r  < 6^-. 
\end{equation}

Define 
$$\|u\|_{S(L^2,I)}=\sup_{(q,r)\in \mathcal{A}_0}\|u\|_{L^q_I L^r_x}$$ and $$\|u\|_{S'(L^2,I)}=\inf_{(q,r)\in \mathcal{A}_0}\|u\|_{L^{q'}_I L^{r'}_x},$$
where $\mathcal{A}_0$ denotes the set of $L^2$-admissible pairs. We also define the following norm
\begin{equation}\label{norm-H2}
    \|\langle \nabla \rangle u \|_{S(L^2,I)}=\|u\|_{S(L^2,I)}+\| \nabla u\|_{S(L^2,I)}.
\end{equation} 
If $I = \R$, $I$ is omitted usually.

\begin{lemma}[Dispersive estimate and Strichartz estimate, \cite{HONG}] Let \( V : \mathbb{R}^3 \to \mathbb{R} \) satisfies \eqref{poten}. Then, the following statements hold:
\begin{equation}\label{eq:dispersive_estimate}
    \| e^{-itH} \|_{L^1 \to L^\infty} \lesssim |t|^{-\frac{3}{2}},
\end{equation}
\begin{align}
\label{SE1}
\qquad \quad\| e^{-itH} f \|_{L_t^q L_x^r}\;\; \lesssim  \|f\|_{L^2},
\end{align}
\begin{align}
\label{SE2} 
\biggl\| \int_0^t e^{-i(t-s)H} g(s)\,ds\biggr\|_{L_t^q L_x^r} &\lesssim \|g\|_{L_t^{m'} L_x^{n'}}, 
\end{align}
for any $L^2$-admisible pairs $(q,r)$ and $(m,n)$.
\end{lemma}


\begin{lemma}\label{hartree_holder}
Let
\begin{equation}\label{eq:hartree_holder}
\frac{1}{r} + \frac{\gamma}{3} = \frac{1}{p} + \frac{1}{q}.
\end{equation}
Then,
\begin{equation}
\Vert (I_\gamma * f)g \Vert_{L_x^r} \leq \Vert f \Vert_{L_x^{p}} \Vert g \Vert_{L_x^{q}}.
\end{equation}
\end{lemma} 
\begin{proof}
Using Hölder and Hardy-Littlewood-Sobolev (Lemma \ref{lemma:Hardy}) inequalities,  
$$\Vert (I_\gamma * f)g \Vert_{L_x^r} \leq \Vert I_\gamma \ast f \Vert_{L_x^{p_1}} \Vert g \Vert_{L_x^q} \leq \Vert f \Vert_{L_x^p} \Vert g \Vert_{L_x^q},$$ where
$$ \frac{1}{r} = \frac{1}{p_1} + \frac{1}{q}, \quad
\text{and}\quad \gamma = \frac{3}{p} - \frac{3}{p_1}.$$
Substituting $p_1$ from the second equation in the first equation we get \eqref{eq:hartree_holder}. 
\end{proof}

\subsection{Local and Global well-posedness}
We start with the local well-posedness. Because of, the equivalence of norms $\| \Lambda u \|_{L^r} \sim \| \nabla u\|_{L^r}$, we need $1<r<3$. 
\begin{proposition}
Let $N = 3$ and $p$ such that $2 \leq p < 1+ \frac{\gamma+2}{N-2}$. Suppose $V$ satisfies conditions \eqref{poten}-\eqref{poten2} and $u_0 \in H^1(\R^N)$. Then there exists $T>0$, $T=T(\|u_0\|_{H^1},N,p,\gamma)$ and there is a unique solution $u$ of \eqref{GHP} in $[0,T]$ such that
$$
u \in C([0,T],H^1(\R^N)) \cap L^{q}([0,T],W_V^{1,r}(\R^N)),
$$    
where $(q,r)$ is $L^2$-admissible.
\end{proposition}
\begin{proof}
 See \cite{ARORA-ROUDENKO-22}. They showed the same result without potential (for $N\geq 3$), using the pair $(q,r)$ such that $2\leq r<N$. The proof here is similar.  
\end{proof}

\ Showing the global existence of solutions in \(H^1\), we need to be more careful. Specifically, for the equivalence of norms \(\|\Lambda u\|_{L^r} \sim \|\nabla u\|_{L^r}\), it is necessary to choose appropriate pairs \((q, r)\). Before stating the global result, we establish the nonlinearity estimate.
\begin{lemma}[Nonlinear estimates]\label{L:NL} Let $p \geq 2$, $\gamma < 3$ and $\frac{5+\gamma}{3}< p < \gamma + 3$. Then, for any $\epsilon>0$ small
\begin{itemize}
\item [(i)] $\left \| (I_\gamma \ast |u|^p) |u|^{p-2} v \right\|_{L_t^{2'} L_x^{6'}} \lesssim \| u \|_{L_t^{\bar a} L_x^{\bar r}}^{2p-2} \| v \|_{L_t^{4^+} L_x^{3^-}}$,
\item [(ii)] $\left \| \nabla ( (I_\gamma \ast |u|^p) |u|^{p-2} u ) \right\|_{L_t^{2'} L_x^{6'}} \lesssim \| u \|_{L_t^{\bar a} L_x^{\bar r}}^{2p-2} \| \nabla u \|_{L_t^{4^+} L_x^{3^-}}$,
\end{itemize}
where
\begin{equation}\label{pairsused}
\bar r=\frac{12(p-1)}{3+2\gamma-2\epsilon}, \quad\bar a=\frac{8(p-1)}{1+2\epsilon}, \quad3^- = \frac{3}{1+\epsilon} \quad\text{and}\quad 4^+ = \frac{4}{1-2\epsilon}.    
\end{equation}
\end{lemma}
\begin{proof}
We start with (i). Note that, $\frac{1}{6'} + \frac{\gamma}{3} = \frac{p}{\bar r} + \frac{p-2}{\bar r} + \frac{1}{3^-}$. Applying Lemma \ref{hartree_holder} and the Hölder inequality. It follows that 
$$\left \| (I_\gamma \ast |u|^p) |u|^{p-2} v \right\|_{L_x^{6'}} \lesssim \| u \|^p_{L_x^{\bar r}} \| u \|_{L_x^{\bar r}}^{p-2} \| v \|_{L_x^{3^-}} = \| u \|_{L_x^{\bar r}}^{2p-2} \| v \|_{L_x^{3^-}}.$$
Now, because of $\frac{1}{2'}=\frac{2p-2}{\bar{a}}+\frac{1}{3^-}$ and by Hölder inequality in the time variable, we have 
$$\| (I_\gamma \ast |u|^p) |u|^{p-2} v \|_{L_t^{2'} L_x^{6'}} \lesssim \| u \|_{L_t^{\bar{a}} L_x^{\bar{r}}}^{2p-2} \| v \|_{L_t^{4^+} L_x^{3^-}}.$$

We now estimate (ii). Observe that
$$ \nabla \left( (I_\gamma \ast |u|^p) |u|^{p-2} u \right) \sim \left( I_{\gamma} \ast |u|^{p-1} \nabla u\right) |u|^{p-2} u  +   \left(I_\gamma \ast |u|^p\right) |u|^{p-2} \nabla u =I+II.$$
As before, by replacing $v$ by $\nabla u$, we obtain II. On the other hand, since
$$\left \| (I_\gamma \ast |u|^{p-1} \nabla u) |u|^{p-2} u ) \right\|_{L_x^{6'}} \lesssim \| u \|^{p-1}_{L_x^{\bar r}} \| \nabla u \|_{L_x^{3^-}} \| u \|_{L_x^{\bar r}}^{p-2} \| u \|_{L_x^{\bar r}}=\| u \|^{2p-2}_{L_x^{\bar r}} \| \nabla u \|_{L_x^{3^-}},$$
and by arguing as (i), we complete the proof.
\end{proof}
\begin{remark}
  Note that, the pair \((\bar{a}, \bar{r})\) which satisfies \eqref{Hs_equation} but not necessarily \eqref{HsAdmissivel}, i.e., it is not \(\dot{H}^{s_c}\)-admissible.
\end{remark}
\begin{proposition} (Global well-possedness) 
Let $p \geq 2$ satisfy $\frac{5+\gamma}{3} < p < 3+ \gamma$ with $0 < \gamma < 3$. Let $V$ satisfy \eqref{poten} and
\eqref{poten2}. Suppose that $\Vert u_0 \Vert_{H^1} \leq E$ and take $\bar r=\frac{12(p-1)}{3+2\gamma-2\epsilon}$, $\bar a=\frac{8(p-1)}{1+2\epsilon}$. Then there exists $\delta = \delta(E)$ such that if
$$ \Vert e^{-itH} u_0 \Vert_{L_t^{\bar a} L_x^{\bar r}} < \delta,$$
then there exist a unique solution $u$ to \eqref{GHP} with initial condition $u_0 \in H^1(\R^3)$ which is globally defined on $[0, \infty)$ such that
$$\Vert u \Vert_{L_t^{\bar a} L_x^{\bar r}} <  2\delta\quad \textnormal{and}\quad \Vert u \Vert_{S(L^2)} + \Vert \Lambda u \Vert_{S(L^2)} \leq 2 C \Vert u_0 \Vert_{H^1}.$$
\end{proposition}

\begin{proof}
For $\rho$ and $M$ to be defined later, define
$$B = \lbrace u \ : \ \Vert u \Vert_{L_t^{\bar a} L_x^{\bar r}} < \rho \text{ and } \Vert u \Vert_{S(L^2)} + \Vert \Lambda u \Vert_{S(L^2)} \leq M \rbrace.$$
We show that the operator $G$ defined by
$$G(u) = e^{-itH} u_0 + i \int_0^t e^{-i(t-s)H} N(u(s)) ds,$$
where $N(u) = (I_{\gamma} \ast |u|^{p}) |u|^{p-2} u$
is a contraction on $B$, with the metric $d(u,v)=\|u-v\|_{S(L^2)}.$ Indeed,
$$ \Vert G(u) \Vert_{L_t^{\bar a} L_x^{\bar r}} \leq \Vert e^{-itH} u_0 \Vert_{L_t^{\bar a} L_x^{\bar r}} + \Vert \int_0^t e^{-i(t-s)H} N(u(s)) ds \Vert_{L_t^{\bar a} L_x^{\bar r}},$$
and using Sobolev inequality, 
$$ \Vert G(u) \Vert_{L_t^{\bar a} L_x^{\bar r}} \leq \Vert e^{-itH} u_0 \Vert_{L_t^{\bar a} L_x^{\bar r}} + \Vert \Lambda^{s_c} \int_0^t e^{-i(t-s)H} N(u(s)) ds \Vert_{L_t^{\bar a} L_x^{\tilde p}},$$
where $s_c = \frac{3}{\tilde p} - \frac{3}{\bar r}$ and $\tilde p= \frac{12(p-1)}{6p-7-2\epsilon}$. We observe the pair $(\bar a,\tilde p)$ is $L^2$-admissible. The Strichartz estimate, equivalence of norms, and interpolation yield
$$\Vert G(u) \Vert_{L_t^{\bar a} L_x^{\bar r}} \leq \Vert e^{-itH} u_0 \Vert_{L_t^{\bar a} L_x^{\bar r}} + c\Vert D^{s_c} N(u) \Vert_{L_t^{2'} L_x^{6'}},$$
$$\qquad \qquad \qquad \qquad \qquad\leq \Vert e^{-itH} u_0 \Vert_{L_t^{\bar a} L_x^{\bar r}} + c\Vert N(u) \Vert_{L_t^{2'} L_x^{6'}}^{1-s_c} \Vert \nabla N(u) \Vert_{L_t^{2'} L_x^{6'}}^{s_c}.$$
Thus, Lemma \ref{L:NL} and taking $u\in B$ imply
$$ \Vert G(u) \Vert_{L_t^{\bar a} L_x^{\bar r}} \leq \Vert e^{-itH} u_0 \Vert_{L_t^{\bar a} L_x^{\bar r}}+ c\Vert u \Vert_{L_t^{\bar a} L_x^{\bar r}}^{(2p-2)(1-s_c)} \Vert u \Vert_{L_t^{4^+} L_x^{3^-}}^{1-s_c} \Vert u \Vert_{L_t^{\bar a} L_x^{\bar r}}^{(2p-2)s_c} \Vert \nabla u \Vert_{L_t^{4^+} L_x^{3^-}}^{s_c}.
$$
\begin{equation*} \label{ones}
  < \delta + c\rho^{(2p-2)(1-s_c)} M^{1-s_c} \rho^{(2p-2)s_c} M^{s_c} = \delta +  c\rho^{2p-2} M.
\end{equation*}

On the other hand, by applying the Strichartz estimates and Lemma \ref{L:NL} once again, we obtain
\begin{eqnarray*}
\Vert G(u) \Vert_{S(L^2)} &\leq& \Vert e^{-itH} u_0 \Vert_{S(L^2)} + c\Vert N(u) \Vert_{S'(L^2)} \leq c\Vert u_0 \Vert_{L^2} + c\Vert N(u) \Vert_{S'(L^2)}\\
 &\leq &c\Vert u_0 \Vert_{L^2} + c\Vert u \Vert_{L_t^{\bar a} L_x^{\bar r}}^{2p-2} \Vert u \Vert_{L_t^{4^+} L_x^{3^-}} .  
\end{eqnarray*}
Similarly (by equivalence of norms)
$$\Vert \Lambda G(u) \Vert_{S(L^2)} \leq c \Vert \nabla u_0 \Vert_{L^2} + c\Vert u \Vert_{L_t^{\bar a} L_x^{\bar r}}^{2p-2} \Vert \nabla u \Vert_{L_t^{4^+} L_x^{3^-}}.$$
Hence,
\begin{equation}\label{twice}
    \Vert G(u) \Vert_{S(L^2)}+\Vert \Lambda G(u) \Vert_{S(L^2)}    \leq c\Vert u_0 \Vert_{H^1} + c \Vert u \Vert_{L_t^{\bar a} L_x^{\bar r}}^{2p-2} \left( \Vert u \Vert_{L_t^{4^+} L_x^{3^-}}+\Vert \nabla u \Vert_{L_t^{4^+} L_x^{3^-}} \right) \leq E + c\rho^{2p-2} M.
\end{equation}
From \eqref{ones} and \eqref{twice}, we take $M=2E$ and
\begin{equation*}
    \rho=\min\left\{ ( \frac{1}{2M})^{\frac{1}{2p-3}},( \frac{1}{2})^{\frac{1}{2p-2}}\right\},
\end{equation*}
we have that $G(u) \in B.$ Now we show that $G$ is a contraction on $B$. for $u, v \in B$,
\begin{eqnarray*}
d(G(u),G(v)) &\lesssim & \|G(u(t))-G(v(t))\|_{S(L^2)} \\
&\lesssim & \| (I_\gamma \ast|u|^p)|u|^{p-2}u - (I_\gamma\ast|v|^p)|v|^{p-2}v \|_{S'(L^2)}.
\end{eqnarray*}
Rewriting, 
$$
d(G(u),G(v)) \lesssim \| (I_\gamma\ast|u|^p) \left( |u|^{p-2}u - |v|^{p-2}v \right) \|_{S'(L^2)} +  \| (I_\gamma\ast\left( |u|^p - |v|^p \right) |v|^{p-2}v \|_{S'(L^2)} := A_1 + A_2.
$$
Thus,
$$
A_1 \lesssim \|u\|_{L_t^{\bar a} L_x^{\bar r}}^p  \left( \|u\|_{L_t^{\bar a} L_x^{\bar r}}^{p-2} + \|v\|_{L_t^{\bar a} L_x^{\bar r}}^{p-2} \right) \|u-v\|_{S(L^2)} \leq 2 M^{2p-2} \|u-v\|_{S(L^2)}.
$$

$$
A_2 \lesssim \left( \|u\|_{L_t^{\bar a} L_x^{\bar r}}^{p-1} + \|v\|_{L_t^{\bar a} L_x^{\bar r}}^{p-1} \right) \|v\|_{L_t^{\bar a} L_x^{\bar r}}^{p-1} \|u-v\|_{S(L^2)} \leq 2 M^{2p-2} \|u-v\|_{S(L^2)}.
$$
For $u,v \in B$ we deduce
$$
d(G(u(t)),G(v(t))) \leq  4c M^{2p-2} d(u,v),
$$
so, $G$ is a contraction.
\end{proof}


\begin{lemma}[{Space-time bounds imply scattering}]
\label{bound-scattering}
Let $p \geq 2$ satisfy $\frac{5+\gamma}{3} < p < 3+ \gamma$ with $0 < \gamma < 3$. Suppose $V$ satisfies conditions \eqref{poten}-\eqref{poten2} and $u$ be a global solution to \eqref{GHP} satisfying $\|u\|_{L^\infty_t H^1_x}\leq E$. If 
\begin{equation*}
 \|u\|_{ L^{\bar a}_{[T,+\infty)} L^{\bar r}_x }<+\infty,
\end{equation*}
for some $T>0$, then
 $u$ scatters forward in time in $H^1$.
\begin{proof} 
The proof is standard; here are the main steps. For $\eta>0$, let $[T,+\infty) = \displaystyle\bigcup_{j=1}^N I_j$, where the intervals $I_j$ are chosen such that $\|u\|_{ L^{\bar a}_{I_j} L^{\bar r}_x }<\eta$ for all $j$. By Strichartz, Lemma \ref{L:NL} and a continuity argument we get 
\begin{equation}
     \|\langle\nabla\rangle u\|_{S(L^2,[T,+\infty))} < +\infty.
\end{equation}
Defining
$$
\phi^+=e^{-itH}u(T)+i\int\limits_{T}^{+\infty} e^{-itH} N(u)(s) ds,
$$

we conclude that \;\;$\|u(t)-e^{-itH}\phi^+\|_{H^1}\lesssim
\| u \|^{2p-2}_{ L^{\bar a}_{[t,+\infty)} L^{\bar r}_x } \| \langle \nabla \rangle u \|_{ L^{4^+}_{[t,+\infty)} L^{3^-}_x }\rightarrow 0, \,\,\textnormal{as}\,\,t\rightarrow +\infty.
$


\end{proof}
\end{lemma}

\section{\bf SCATTERING CRITERION}\label{sec:criterion}

\ In this section, we study the scattering criterion for the generalized Hartree equation with a potential.

\begin{proposition}[Scattering criterion]\label{scattering_criterion}
Assume the same hypotheses as in Theorem \ref{theo:scattering}. Consider an $H^1$-solution $u$ to \eqref{GHP} defined on $[0,+\infty)$ and assume 
\begin{equation}\label{E}
\displaystyle\sup_{t \in [0,+\infty)}\left\|u(t)\right\|_{H^1_x} := E < +\infty.
\end{equation}

There exist constants $R > 0$ and $\varepsilon>0$ depending only on $E$, $p$ (but never on $u$ or $t$) such that if
\begin{equation}\label{scacri}
\liminf_{t \rightarrow +\infty}\int_{B(0,R)}|u(x,t)|^2 \, dx \leq \varepsilon^2,
\end{equation}
then $u$ scatters forward in time in $H^1$.
\end{proposition}

\ To this end, we begin with the following lemma, which is fundamental to the proof.  Since the pair $(\bar{a},\bar{r})$ is not $\dot{H}^{s_c}$-admissible, careful attention is required when estimating the recent past $F_1$ and distant past $F_2$ (see the proof).

\begin{lemma}\label{linevo}
Let $p\geq 2$, $0<\gamma<3$ such that $\frac{5+\gamma}{3} < p < \gamma + 3$. Suppose $V$ satisfies conditions \eqref{poten}-\eqref{poten2} and $u$ be a radial $H^1$-solution to \eqref{GHP} satisfying \eqref{E}. If $u$ satisfies \eqref{scacri} for some $0 < \varepsilon < 1$, then there exist $\gamma, T > 0$ such that 
\begin{equation}\label{norm-small}
\left\| e^{-i(\cdot-T)H} u(T)\right\|_{L^{\bar a}_{[T,+\infty)} L^{\bar r}_x}  \lesssim \varepsilon^\gamma.
\end{equation}
\end{lemma}

\begin{proof}

Fix the parameters $\mu, \gamma >0$ (to be chosen later). Applying Strichartz estimate \eqref{SE1}, there exists $T_{0} > \varepsilon^{-\mu}$ such that
\begin{equation}
\label{T0}
\left\| e^{-itH} u_0\right\|_{L^{\bar a}_{[T,+\infty)} L^{\bar r}_x} \leq \varepsilon^{\gamma}.
\end{equation}

For $T\geq T_0$, define  $I_1 :=\left[T-\varepsilon^{-\mu}, T\right]$ and  $I_2 := [0, T-\varepsilon^{-\mu}]$. Duhamel's formula
\begin{equation}
u(T) = e^{-itH}u_0 - i\int_0^{T} e^{-i(t-s)H} N(u)(s) \, ds,
\end{equation}
implies that
\begin{equation}
e^{-i(t-T)H}u(T)  = e^{-itH} u_0 - iF_1 - iF_2,
\end{equation}
where, for $i = 1,2,$
$$
F_i = \int_{I_i} e^{-i(t-s)H} N(u) (s) \,ds.
$$
We refer to $F_1$ as the ``recent past" and to $F_2$ as the ``distant past". By \eqref{T0}, it remains to estimate $F_1$ and $F_2$.

\textbf{Step 1. Estimate on recent past.}

Let $\eta$ denote a smooth, spherically symmetric function which equals $1$ on $B(0, 1/2)$ and $0$ outside $B(0,1)$. For any $R > 0$ we use $\eta_R$ to denote the rescaling $\eta_R(x) := \eta(x/R)$.

By hypothesis \eqref{scacri}, we can fix $T\geq T_0$ such that
\begin{equation}\label{mass}
\int \eta_R(x)\left|u(T,x)\right|^2dx\lesssim \varepsilon^2.
\end{equation}

The following relation is obtained by multiplying \eqref{GHP} by $\eta_R\bar{u}$, taking the imaginary part and integrating by parts (see Tao \cite[Section 4]{Tao} for the NLS version),

$$
\partial_t\int \eta_R|u|^2\, dx = 2\Im\left(\int \nabla\eta_R \cdot \nabla u \bar{u} \right).
$$
Thus, using \eqref{E},
$$
\left| \partial_t \int \eta_R(x)|u(t,x)|^2dx\right| \lesssim \frac{1}{R},
$$

so that,  by \eqref{mass}, for $t \in I_1$,
\begin{equation*}
    \int \eta_R(x)\left|u(t,x)\right|^2dx\lesssim \varepsilon^2+\frac{\varepsilon^{-\mu}}{R}.
\end{equation*}

If $R > \varepsilon^{-(\mu+2)}$, then we have
\begin{equation}
\label{eq:etaR_u}
 \left\| \eta_Ru\right\|_{L^\infty_{I_1}L^2_x} \lesssim 
\varepsilon.   
\end{equation}

Recall the pairs used in Section 2. $(\bar a,\bar r)$ satisfies \eqref{Hs_equation}, $(\bar a, \tilde p)$, $(2,6)$ and $(4^+,3^-)$ are $L^2$-admissible.

$$ \| F_1(u) \|_{L_t^{\bar a} L_x^{\bar r}} \lesssim \| D^{s_c} F_1(u) \|_{L_t^{\bar a} L_x^{\bar p}} \lesssim \| D^{s_c} N(u) \|_{L_t^{2'} L_x^{6'}} \lesssim \| u \|_{L_t^{\bar a} L_x^{\bar r}}^{2p-2} \| D^{s_c} u \|_{L_t^{4^+} L_x^{3^-}}. $$

Using interpolation, 
\begin{equation}\label{interpolation}
\| u \|_{L_t^{\bar a} L_x^{\bar r}} \leq \|u\|_{L_t^q L_x^r}^{1-s_c} \|u\|_{L_t^m L_x^n}^{s_c},  
\end{equation}
where $(q,r)$ which is $L^2$-admissible, $(m,n)$  is $S(\dot{H}^{1})$-admissible and
$$\frac{1}{\bar a} = \frac{1-s_c}{q} + \frac{s_c}{m} \qquad \text{and} \qquad \frac{1}{\bar r} = \frac{1-s_c}{r} + \frac{s_c}{n}.$$
Defining the numbers 
$$
q=\frac{8(\gamma+2)}{3(1+2\epsilon)}, \qquad 
r=\frac{4(\gamma+2)}{3+2\gamma-2\epsilon}, \qquad m=3q,\;\;\text{and}\;\; n=3r, 
$$
we obtain \eqref{interpolation} (using the value of\footnote{that, $(q,r)$ is $L^2$- admissible and $(m,n)$ is $S(\dot{H}^1)$-admissible. Moreover, $2\leq r\leq 6$ and $m\geq 2$.} $\bar{a}$). By Sobolev embedding, with $\frac{1}{n} = \frac{1}{s} - \frac{1}{3}$, we have
\begin{equation}\label{interp2}
 \| u \|_{L_t^{\bar a} L_x^{\bar r}} \leq \|u\|_{L_t^q L_x^r}^{1-s_c} \|\nabla u\|_{L_t^m L_x^s}^{s_c}.
\end{equation}
Thus $n=\frac{3s}{3-s}$, or equivalently $s=\frac{3n}{3+n}$. Observe that $(m,s)$ is $L^2$-admissible and $2\leq s<3$.
\ Combining the relation \eqref{interpolation}, \eqref{interp2} and interpolation, it follows that
\begin{align}\label{eq:F1_bounds}
 \| F_1(u) \|_{L_t^{\bar a} L_x^{\bar r}} &\lesssim \| u \|_{L_t^q L_x^r}^{(1-s_c)(2p-2)} \| u \|_{L_t^m L_x^n}^{s_c(2p-2)} \| D^{s_c} u \|_{ L_t^{4^+} L_x^{3^-} }\nonumber\\
 &\lesssim \| u \|_{L_t^q L_x^r}^{(1-s_c)(2p-2)} \| \nabla u \|_{L_t^m L_x^s}^{s_c(2p-2)} \| D^{s_c} u \|_{ L_t^{4^+} L_x^{3^-} } \nonumber\\
 &\lesssim \| u \|_{L_t^q L_x^r}^{(1-s_c)(2p-2)} \| \nabla u \|_{L_t^m L_x^s}^{s_c(2p-2)} \|u\|_{ L_t^{4^+} L_x^{3^-} }^{1-s_c} \| \nabla u \|_{L_t^{4^+} L_x^{3^-} }^{s_c}.
 \end{align}
Lemma \ref{bound-interval-nabla_u} gives
\begin{equation}
\label{eq:tree-factors}
\| \nabla u \|_{L_t^m L_x^s}^{s_c(2p-2)} \|u\|_{ L_t^{4^+} L_x^{3^-} }^{1-s_c} \| \nabla u \|_{L_t^{4^+} L_x^{3^-} }^{s_c} \lesssim  (\varepsilon^{-\mu})^{\frac{s_c(2p-2)}{m}} (\varepsilon^{-\mu})^{\frac{1-s_c}{4^+}} E^{1-s_c} (\varepsilon^{-\mu})^{\frac{s_c}{4^+}}.    
\end{equation}

It remains to study $\| u \|_{L_t^q L_x^r}$. The Hölder inequality and interpolation imply 
\begin{eqnarray*}
 \| u \|_{L_t^q L_x^r} &\leq& \| 1 \|_{L_t^q} \| u \|_{L_t^\infty L_x^r} \leq \epsilon^{-\frac{\mu}{q} } \| u \|_{L_t^\infty L_x^r}\leq \epsilon^{-\frac{\mu}{q} }(\| \eta_R u \|_{L_t^\infty L_x^r} + \| (1-\eta_R) u \|_{L_t^\infty L_x^r})\\ 
 &\leq& \epsilon^{-\frac{\mu}{q} }\| \eta_R u\|_{L_t^\infty L_x^2}^{1-\theta} \|u\|_{L_t^\infty L_x^6}^\theta +\epsilon^{-\frac{\mu}{q} } \| (1-\eta_R) u \|_{L_t^\infty L_x^\infty}^{1-\bar \theta} \|u\|_{L_t^\infty L_x^2}^{\bar \theta},
\end{eqnarray*}
where
$$\frac{1}{r} = \frac{1-\theta}{2} + \frac{\theta}{6} \qquad \text{and} \qquad \frac{1}{r} = \frac{1-\bar \theta}{\infty} + \frac{\bar \theta}{2},$$
thus
$$ \theta = \frac{3(r-2)}{2r} = \frac{3(1+2\epsilon)}{4(\gamma+2)} \qquad \text{and} \qquad \bar \theta = \frac{2}{r} = \frac{2\gamma+3-2\epsilon}{2(\gamma+2)}. $$

The hypothesis \eqref{E} leads to,
$$ \| u \|_{L_t^q L_x^r} \lesssim \varepsilon^{-\frac{\mu}{q}} \left( \| \eta_R u\|_{L_t^\infty L_x^2}^{1-\theta} + \| (1-\eta_R) u \|_{L_t^\infty L_x^\infty}^{1-\bar \theta} \right),
$$
which, using \eqref{eq:etaR_u} and Lemma \ref{radial-sobolev}, give
$$ \| u \|_{L_t^q L_x^r} \lesssim \varepsilon^{-\frac{\mu}{q}} (\varepsilon^{1-\theta} + R^{-\frac{1-\bar \theta}{2}}).$$

Since $R>\varepsilon^{-2-\mu}$,
$$ \| u \|_{L_t^q L_x^r} \lesssim \varepsilon^{-\frac{\mu}{q}} (\varepsilon^{1-\theta} + \varepsilon^{(1-\bar \theta)\frac{2+\mu}{2}}). 
$$
Note that $\theta < \bar \theta$ and choosing $\mu$ sufficiently small, one has $1-\theta > (1-\bar \theta)\frac{2+\mu}{2}$, so
\begin{equation}
\label{eq:u_q_r}
\| u \|_{L_t^q L_x^r} \lesssim \varepsilon^{-\frac{\mu}{q}} \varepsilon^{(1-\bar \theta)\frac{2+\mu}{2}} = \varepsilon^{(1-\bar \theta)\frac{2+\mu}{2} - \frac{\mu}{q}}.
\end{equation}
Finally, applying \eqref{eq:tree-factors} and \eqref{eq:u_q_r} we deduce that
$$
\| F_1(u) \|_{L_t^{\bar a} L_x^{\bar r}} \lesssim (\varepsilon^{(1-\bar \theta)\frac{2+\mu}{2} - \frac{\mu}{q}})^{(1-s_c)(2p-2)}  \varepsilon^{ - \mu\frac{s_c(2p-2)}{m}} \varepsilon^{-\mu \frac{1-s_c}{4^+}} \varepsilon^{-\mu \frac{s_c}{4^+}}.
$$
Since $\mu$ can be chosen very small, the exponent of $\varepsilon$ is positive, that is
$$\| F_1(u) \|_{L_t^{\bar a} L_x^{\bar r}} \lesssim \varepsilon^{\nu_1}.$$

\textbf{Step 2. Estimate on distant past.}


We have
\begin{equation}
\label{eq:F2_inicio}
\left\|F_2\right\|_{L_{[T,+\infty)}^{\bar a}L_x^{\bar r}}  \leq 
\left\|F_2\right\|^{1-\theta}_{L_{[T,+\infty)}^{k}L_x^{l}}
\left\|F_2\right\|^{\theta}_{L_{[T,+\infty)}^{\bar p}L_x^{\infty}},   
\end{equation}
where $0 < \theta < 1$ and
\begin{equation}
\label{interp_kl_pinfty}
\frac{1}{\bar r} = \frac{\theta}{l} \qquad \text{and} \qquad \frac{1}{\bar a} = \frac{\theta}{k} + \frac{1-\theta}{\bar p},    
\end{equation}
such that $(k,l)$ is $L^2$-admissible and $\bar p >2$.

We can rewrite $F_2$ applying Duhamel's principle as
$$
F_2 = e^{-itH}\left[e^{-i(-T+\epsilon^{-\mu})H}U(-T+\epsilon^{-\mu})u(T-\epsilon^{-\mu})-u(0)\right].
$$
Using Strichartz estimate \eqref{SE1} in \eqref{eq:F2_inicio} leads to
\begin{equation}
\label{eq:F2_La_Lr}
\begin{split}
\left\| F_2\right\|_{L_{[T,+\infty)}^{\bar a}L_x^{\bar r}}
&\leq\left\| e^{-itH}\left[U(-T+\epsilon^{-\mu})u(T-\epsilon^{-\mu})-u(0)\right] \right\|^{1-\theta}_{L_{[T,+\infty)}^{k}L_x^{l}}
\left\| F_2 \right\|^{\theta}_{L_{[T,+\infty)}^{\bar p}L_x^{\infty}}\\
&\lesssim\left(\left\|u\right\|_{L^\infty_t L^2_x}\right)^{1-\theta} \left\|F_2\right\|^{\theta}_{L_{[T,+\infty)}^{\bar p} L_x^{\infty}}
\lesssim \left\|F_2\right\|^{\theta}_{L_{[T,+\infty)}^{\bar p} L_x^{\infty}}.
\end{split}
\end{equation}

For $t>T$ and using the dispersive estimate \eqref{eq:dispersive_estimate}, we have
$$
\|F_2(t)\|_{L_x^\infty} \leq \int_0^{T-\varepsilon^{-\mu}} \|e^{-i(t-s)H} N(u(s))\|_{L_x^\infty} ds \lesssim \int_0^{T-\varepsilon^{-\mu}} (t-s)^{-\frac{3}{2}} \|N(u(s))\|_{L_x^1} ds. 
$$
Using Lemma \ref{hartree_holder},
$$
\|N(u(s))\|_{L^1} \leq \|u(s)\|_{L_x^{p_1}}^p \|u(s)\|_{L_x^{p_1}}^{p-1},
$$
where $1+\frac{\gamma}{3}=\frac{p}{p_1} + \frac{p-1}{p_1}$. That is, $p_1=\frac{3(2p-1)}{3+\gamma}$. Since $\frac{5+\gamma}{3}<p<3+\gamma$, it is easy to see that $2<p_1<6$. Thus, by the Sobolev inequality and the hypothesis \eqref{E},
$$
\|F_2(t)\|_{L_x^\infty} \lesssim \int_0^{T-\epsilon^{-\mu}} (t-s)^{-\frac{3}{2}} ds \lesssim (t-T+\varepsilon^{-\mu})^{-\frac{1}{2}}.
$$
Moreover,
$$
\|F_2\|_{L_{[T,+\infty)}^{\bar p} L_x^{\infty}} \lesssim \left( \int_T^\infty (t-T+\varepsilon^{-\mu})^{-\frac{\bar p}{2}} dt \right)^{\frac{1}{\bar p}} \lesssim \varepsilon^{\mu\left(\frac{1}{2}-\frac{1}{\bar p}\right)},
$$
provided that $\bar p>2$. Thus, from \eqref{eq:F2_La_Lr} we get
$$
\|F_2\|_{L_{[T,+\infty)}^{\bar a}L_x^{\bar r}} \lesssim \varepsilon^{\mu\left(\frac{1}{2}-\frac{1}{\bar p}\right)\theta} =: \varepsilon^{\nu_2}.
$$

\ Now, we choose $\theta \in (0,1)$ and $k,l,\bar p$ satisfying \eqref{interp_kl_pinfty}, such that $(k,l)$ is $L^2$-admissible and $\bar p >2$.

Since $(k,l)$ is $L^2$-admissible and $(\bar a,\bar r)$ verifies \eqref{Hs_equation} we get (using \eqref{interp_kl_pinfty})
$$
\bar p = \frac{4(1-\theta)(p-1)}{2+\gamma-3\theta(p-1)}.
$$
The condition $\bar p>2$ implies $\theta > \frac{\gamma+4-2p}{p-1}$, so $p<\frac{\gamma+4}{2}$. Note that, $l=\bar r \theta \in [2,6]$, or 
$$\frac{2}{\bar r} \leq \theta <  \frac{6}{\bar r},$$
and $k=\frac{4l}{3l-6}\geq 2.$ Thus, if $\frac{5+\gamma}{3} < p < \frac{\gamma+4}{2}$, we take $\theta = \left(\frac{\gamma+4-2p}{p-1}\right)^+$. If $\frac{\gamma+4}{2}\leq p<\gamma+3$, we take $\theta = \frac{2}{\bar r}$, which implies that $l=2$ and $k=\infty$.

\ Finally, defining $\gamma := \min\{ \nu_1, \nu_2 \}$, recalling \eqref{T0} and 
$
e^{-i(t-T)H}u(T) = e^{-itH}u_0 + iF_1 + iF_2,
$
we obtain \eqref{norm-small}.

\end{proof}

We now show the following lemma which was used in Step 1, Eq. \eqref{eq:tree-factors}, of the previous result.

\begin{lemma}
\label{bound-interval-nabla_u}
Let $u$ be a solution of \eqref{GHP} satisfying
$$ \sup_{t \in [0,\infty)} \|u(t)\|_{H^1} \leq E,$$
and $(q,r)$, $(m,s)$ be the $L^2$-admissible pairs used in Lemma \ref{linevo}.
Then for any $T \geq 0$ and $\tau>0$
\begin{equation}
    \|\nabla u\|_{L_t^q L_x^r ([T,T+\tau]\times \R^3)}^q \lesssim \langle \tau \rangle.
\end{equation}
\end{lemma}

\begin{proof}
The proof follows a standard argument, which can be found in \cite{ARORA-GH-19} and  \cite{CAZENAVEBOOK}.
We take the value $u(T)$ as the initial data, with the solution being well-defined on the interval $I_\tau = [T,T+\tau]$.

\ For any $T>0$ and $\tau>0$, define
$$
X_\tau = \{ u \in L_{I_\tau}^\infty H_x^1 : \nabla u \in L_{I_\tau}^m L_x^s \cap L_{I_\tau}^{4^+} L_x^{3^-}, u \in L_{I_\tau}^q L_x^r \cap L_{I_\tau}^{4^+} L_x^{3^-}  \}.
$$

First suppose $\tau$ sufficiently small so that $\|u\|_{X_\tau} < \beta < 1$. By the computation performed in \eqref{eq:F1_bounds},
$$\|u\|_{X_\tau} \lesssim \|u(T)\|_{H^1} + \| u \|_{L_t^q L_x^r}^{(1-s_c)(2p-2)} \| \nabla u \|_{L_t^m L_x^s}^{s_c(2p-2)} \|u\|_{ L_t^{4^+} L_x^{3^-} }^{1-s_c} \| \nabla u \|_{L_t^{4^+} L_x^{3^-} }^{s_c},
$$
and using the definition of $X_\tau$,
$$
\|u\|_{X_\tau} \lesssim \|u(T)\|_{H^1} + \|u\|_{X_\tau}^{(1-s_c)(2p-1)} \|u\|_{X_\tau}^{s_c(2p-1)} = \|u(T)\|_{H^1} + \|u\|_{X_\tau}^{2p-1}.
$$
Using the smallness of $\|u\|_{X_\tau}$, we have that $\|u\|_{X_\tau} \leq c_{N,p}\|u(T)\|_{H^1}$. 
Subdividing the interval $[T,T+\tau]$ into intervals $I_j$ of length $\tau_0$ such that $\|u\|_{X_{I_j}} \lesssim \beta$, where $\beta=\beta(E)$.
Let \( k = \left\lceil \frac{\tau}{\tau_0} \right\rfloor  \). We use the above argument on each $I_j$, obtaining $\|u\|_{X_{I_j}} \lesssim \|u(T+j \tau_0)\|_{H^1(\R^N)}$. Adding all the pieces, we obtain
$$
\|\nabla u\|_{L_{I_\tau}^q L_x^r}^q \lesssim \sum_{j=0}^{k+1} \int_{T+j\tau_0}^{T+(j+1)\tau_0} \|\nabla u\|_{L_x^r}^q d\tau \lesssim \tau_0 (1+k) \lesssim \langle \tau \rangle.
$$
\end{proof}

\begin{proof}[\bf Proof of Proposition \ref{scattering_criterion}]

Choose $\varepsilon$ is small enough so that, by Lemma \ref{linevo}, $$\left\| e^{-i\cdot H} u(T)\right\|_{L^{\bar a}_{[0,+\infty)} L^{\bar r}_x}  = \left\| e^{-i(\cdot-T)H}  u(T)\right\|_{L^{\bar a}_{[T,+\infty)} L^{\bar r}_x} \lesssim \varepsilon^\gamma\leq \delta_{sd},
$$
which implies that the norm $\|u\|_{L^{\bar a}_{[T,+\infty)} L^{\bar r}_x}$ is bounded, by Theorem  (global well-possedness). Thus, using Lemma \ref{bound-scattering} we conclude that $u$ scatters forward in time in $H^1$.

\end{proof}


\section{\bf VARIATIONAL ANALYSIS}\label{sec:variational}

We recall a Gagliardo-Nirenberg type inequality (when $N=3$), see \cite{ARORA-ROUDENKO-22}. Define $A$ and $B$ such that $A+B=2p$ by
$$A = 3+\gamma-p \qquad \qquad B = 3p-(3+\gamma) .$$
A direct computation shows that 
\begin{equation}
\label{eq:AB_sigmac}
A+2\sigma_c = B\sigma_c,\quad \textnormal{with}\quad \sigma_c = \frac{1-s_c}{s_c}.
\end{equation}

\begin{lemma}[\cite{ARORA-GH-19}, \cite{ARORA-ROUDENKO-22}]
If $p\geq 2$ and $0<\gamma<N$, then the potential $P(u)$ verifies 
$$
P(u) \leq C_{op} \|u\|_2^A \|\nabla u\|_2^B.
$$
Moreover, the equality is attained on ground state solutions $Q$, which solve
$$
-Q + \Delta Q + (I_\gamma \ast |Q|^p) |Q|^{p-2} Q = 0,
$$
where
$$
C_{op} = \frac{2p}{B} \left(\frac{A}{B} \right)^{\frac{B-2}{2}} \frac{P(Q)}
{\|Q\|^{-2(p-1)}_2}.
$$
\end{lemma}

Moreover, the Pohozaev identities are used 
\begin{equation}
\label{eq:pohozaev}
E_0(Q) = \frac{B-2}{2B} \|\nabla Q\|_2^2 = \frac{B-2}{2A} \|Q\|_2^2,\quad P(Q) = \frac{2p}{B} \|\nabla Q\|_2^2,\quad C_{op} = \frac{P(Q)}{\|Q\|_2^A \|\nabla Q\|_2^B}.
\end{equation}

\begin{proposition}(Coercivity)
\label{prop:sem_chi}
Let $u \in H^1(\R^N)$. If there exists $\delta > 0$ such that 
\begin{equation}
\label{eq:Pu_Mu_PQ_MQ}
P(u) M(u)^{\sigma_c} < (1-\delta) P(Q) M(Q)^{\sigma_c},   
\end{equation}
then there is constant $\delta'>0$ (depending on $\delta$) such that
$$
\|\nabla u\|_2^2 - \frac{B}{2p} P(u) \geq \delta' P(u).
$$
\end{proposition}

\begin{proof}
By the definition of the Sobolev optimal constant
$P(u) \leq C_{op} \|u\|^A_2 \|\nabla u\|^B_2$. Thus (by \eqref{eq:AB_sigmac})
\begin{equation}
\begin{split}
\label{eq:Pu_1}
P(u)^{\frac{B}{2}} &\leq C_{op} P(u)^{\frac{B}{2}-1} M(u)^{\frac{A}{2}} \|\nabla u\|_2^B \\
&\leq C_{op} \left( P(u) M(u)^{\sigma_c} \right)^{\frac{B}{2}-1} \|\nabla u\|_2^B.
\end{split}
\end{equation}
On the other hand, since $C_{op} = \frac{P(Q)}{\|Q\|_2^A \|\nabla Q \|_2^B}$, we have
$$
C_{op} = \frac{P(Q)}{M(Q)^{\frac{A}{2}} \left( \frac{B}{2p} P(Q) \right)^{\frac{B}{2}}} = \left( \frac{2p}{B} \right)^{\frac{B}{2}} \frac{1}{M(Q)^{\frac{A}{2}} P(Q)^{\frac{B}{2}-1}}.
$$
Using \eqref{eq:AB_sigmac} again, we get
$
C_{op} = \left( \frac{2p}{B} \right)^{\frac{B}{2}} \frac{1}{\left( M(Q)^{\sigma_c} P(Q) \right)^{\frac{B}{2}-1}}.
$
Substituting this in \eqref{eq:Pu_1},
$$
P(u)^{\frac{B}{2}} \leq  \left( \frac{P(u) M(u)^{\sigma_c}}{P(Q) M(Q)^{\sigma_c}} \right)^{\frac{B}{2}-1} \left(\frac{2p}{B} \|\nabla u\|^2_2 \right)^{\frac{B}{2}},
$$
which implies (by \eqref{eq:Pu_Mu_PQ_MQ}),
$
P(u) \leq \frac{2p}{B} (1-\delta)^{\frac{B-2}{B}} \|\nabla u\|_2^2.
$
Hence,
\begin{equation*}
    \|\nabla u\|_2^2 \geq \frac{B}{2p} \frac{1}{(1-\delta)^{\frac{B-2}{B}}} P(u).
\end{equation*}
Therefore,
$$
\|\nabla u\|_2^2 - \frac{B}{2p} P(u) \geq \frac{B}{2p} \left( \frac{1}{(1-\delta)^{\frac{B-2}{B}}} - 1 \right) P(u)=\delta'P(u).
$$
\end{proof}

\ We also obtain the localized version of Proposition \ref{prop:sem_chi}. Let $\chi(x)$ denote a radial function with support on the ball $|x|\leq 1$, with value identically $1$ on the ball $|x|\leq \frac{1}{2}$ and smoothly decreasing on $\frac{1}{2} \leq |x| \leq 1$. For $R>0$ define the function $\chi_R$ such that $\chi_R(x) = \chi(\frac{x}{R})$. 

\begin{proposition}(Coercivity on balls)
\label{prop:com_chi}
Let $u \in H^1(\R^N)$. If there exists $\delta > 0$ such that 
If $P(u) M(u)^{\sigma_c} < (1-\delta) P(Q) M(Q)^{\sigma_c}$, then there is constant $\delta'>0$ (depending on $\delta$) such that
$$
\int |\nabla (\chi_R u)|^2 dx - \frac{B}{2p} P(\chi_R u) \geq \delta' P(\chi_R u).
$$
\end{proposition}

\begin{proof}
Observe that $\|\chi_R u\|^2_{L^2} \leq  \|u\|^2_{L^2}$ and $P(\chi_R u) \leq P(u)$, using the hypothesis we have that
$$
P(\chi_R u) M(\chi_R u)^{\sigma_c} < (1-\delta)P(Q)M(Q)^{\sigma_c}.
$$
Applying Proposition \ref{prop:sem_chi}, we obtain $ \int |\nabla (\chi_R u)|^2 dx - \frac{B}{2p} P(\chi_R u) \geq \delta' P(\chi_R u).$
\end{proof}

\ Now, we show Corollary \ref{two_conds_one_cond-intro}. It is a consequence of our theorem. To this end, we use the scattering conditions below the threshold to imply the unique condition \eqref{eq:only_one_hypo}. Thus, by Theorem \ref{theo:scattering}, we obtain the desired result.



\begin{proof}[\bf Proof of Corollary \ref{two_conds_one_cond-intro}]

The optimal Sobolev constant is
$
C_{op} = \frac{P(Q)}{\|Q\|^A \|\nabla Q\|^B},
 $
 where $P(Q) = \frac{2p}{B} \|\nabla Q\|_2^2.$ Thus, using Pohozaev inequalities,
 \begin{equation}
 \label{eq:sobolev_constant_2}
     C_{op} = \frac{2p}{B} \frac{\|\nabla Q\|_2^2}{\|Q\|_2^A \|\nabla Q\|_2^B} = \frac{2p}{B\|Q\|_2^A \|\nabla Q\|_2^{B-2}}.
 \end{equation}

By condition \eqref{eq:cond1_ME}, there is a $\delta>0$ such that
\begin{equation}
\label{eq:ME_delta}
M(u)^{\sigma_c} E(u) < (1-\delta) M(Q)^{\sigma_c} E_0(Q).
\end{equation}
Thus,
\begin{align*}
    M(u)^{\sigma_c} E(u) &= \|u\|_2^{2\sigma_c} \frac{1}{2} \|\Lambda u\|_2^2 - \frac{1}{2p} P(u(t)) \|u\|_2^{2\sigma_c} \\
    &\geq \frac{1}{2} \left(  \|\Lambda u \|_2 \|u\|_2^{\sigma_c}  \right)^2 - \frac{C_{op}}{2p} \|\nabla u\|_2^B \|u\|_2^{A+2\sigma_c}.
\end{align*}
Since $V \geq 0$, $\|\nabla u\|_2^2 \leq \|\Lambda u\|_2^2$ and applying \eqref{eq:AB_sigmac}, it follows that
\begin{align*}
    M(u)^{\sigma_c} E(u) &\geq \frac{1}{2} \left(  \|\Lambda u \|_2 \|u\|_2^{\sigma_c}  \right)^2 - \frac{C_{op}}{2p} \|\Lambda u\|_2^B \|u\|_2^{A+2\sigma_c} \\
    &= \frac{1}{2} \left(  \|\Lambda u \|_2 \|u\|_2^{\sigma_c}  \right)^2 - \frac{C_{op}}{2p} \left( \|\Lambda u\|_2 \|u\|_2^{\sigma_c} \right)^B.
\end{align*}
If we define the function $g(x) = \frac{1}{2} x^2 - \frac{C_{op}}{2p} x^B$, then the relation \eqref{eq:ME_delta} implies
\begin{equation}
\label{eq:1-delta}
(1-\delta) M(Q)^{\sigma_c}E_0(Q) > g( \|\Lambda u \|_2 \|u\|_2^{\sigma_c} ).    
\end{equation}
The function $g$ has critical points in $0$ and $x_0=\|Q\|_2^{\sigma_c} \|\nabla Q\|_2$ and $g(x_0)=M(Q)^{\sigma_c} E_0(Q)$. Therefore,
$$
(1-\delta) g(\|Q\|_2^{\sigma_c} \|\nabla Q\|_2) > g( \|\Lambda u \|_2 \|u\|_2^{\sigma_c} ).
$$
Since $g$ is strictly increasing in $(0,x_0)$ and by \eqref{eq:cond2}, one has 
$$
 \|Q\|_2^{\sigma_c} \|\nabla Q\|_2 > \|\Lambda u \|_2 \|u\|_2^{\sigma_c},
$$
which implies \eqref{eq:sup_u_Lambda_u}, that is, the solution $u$ is global.

\ On the other hand, by Pohozaev inequalities \eqref{eq:pohozaev}, we have 
$
E_0(Q) M(Q)^{\sigma_c} = \frac{B-2}{2B} \left( \|\nabla Q\|_2 \|Q\|_2^{\sigma_c} \right)^2.
$
Thus \eqref{eq:1-delta} gives
$$
1-\delta > \frac{B}{B-2} \left( \frac{\|\Lambda u \|_2 \|u\|_2^{\sigma_c}}{\|\nabla Q\|_2 \|Q\|_2^{\sigma_c}} \right)^2 - C_{op} \frac{B}{p(B-2)} \frac{1}{(\|\nabla Q\|_2 \|Q\|_2^{\sigma_c})^2} \left( \|\Lambda u \|_2 \|u\|_2^{\sigma_c}  \right)^B,
$$
and using the expression of $C_{op}$ in \eqref{eq:sobolev_constant_2} and the identity \eqref{eq:AB_sigmac},
$$
\frac{C_{op} B}{p(B-2)} \frac{1}{(\|\nabla Q\|_2 \|Q\|_2^{\sigma_c})^2} = \frac{2}{B-2} \frac{1}{ \|\nabla Q\|_2^B \|Q\|_2^{A+2\sigma_c} } = \frac{2}{B-2} \frac{1}{\left(  \|\nabla Q\|_2 \|Q\|_2^{\sigma_c}  \right)^B}.
$$
Thus,
$$
1-\delta > \frac{B}{B-2} \left( \frac{\|\Lambda u \|_2 \|u\|_2^{\sigma_c}}{\|\nabla Q\|_2 \|Q\|_2^{\sigma_c}} \right)^2 - \frac{2}{B-2} \left( \frac{\|\Lambda u \|_2 \|u\|_2^{\sigma_c}}{\|\nabla Q\|_2 \|Q\|_2^{\sigma_c}} \right)^B.
$$
Denoting $y(t) = \frac{\|\Lambda u \|_2 \|u\|_2^{\sigma_c}}{\|\nabla Q\|_2 \|Q\|_2^{\sigma_c}}$. We have $y(t)<1$ by the first part. Observe that,
$$
1-\delta > \frac{B}{B-2} y(t)^2 - \frac{2}{B-2} y(t)^B.
$$
The function $f(y) = \frac{B}{B-2} y^2 - \frac{2}{B-2} y^B$, verifies $f'(y) = \frac{2B}{B-2} y (1-y^{B-2})$, and thus is strictly increasing in $(0,1)$ and $f(0)=0$  is a local minimum and $f(1) = 1$ is a local maximum.

The condition $\frac{B}{B-2} y^2 - \frac{2}{B-2} y^B < 1-\delta$ for some $\delta>0$, implies by continuity that there exists $\delta'>0$ such that
$y(t) < 1-\delta'$. That is,
$$
\|\Lambda u \|_2 \|u\|_2^{\sigma_c} < (1-\delta') \|\nabla Q\|_2 \|Q\|_2^{\sigma_c}.
$$
Note that
$$
P(u) M(u)^{\sigma_c} \leq C_{op} \|u\|_2^{A+2\sigma_c} \|\nabla u\|_2^B,
$$
where
$
C_{op} = \frac{P(Q)}{\|Q\|_2^A \|\nabla Q\|_2^B} = \frac{P(Q) M(Q)^{\sigma_c}}{\|Q\|_2^{A+2\sigma_c} \|\nabla Q\|_2^B}.
$
The relation \eqref{eq:AB_sigmac} yields
$$
P(u) M(u)^{\sigma_c} \leq \frac{P(Q) M(Q)^{\sigma_c}}{\|Q\|_2^{B\sigma_c} \|\nabla Q\|_2^B} \left( \|u\|_2^{\sigma_c} \|\nabla u\|_2 \right)^B.
$$
Thus
$
\frac{P(u) M(u)^{\sigma_c}}{P(Q) M(Q)^{\sigma_c}} \leq \left( \frac{\|u\|_2^{\sigma_c} \|\nabla u\|_2}{\|Q\|_2^{\sigma_c} \|\nabla Q\|_2} \right)^B < (1-\delta')^B,
$
so
$$
\sup_{t \in [0,\infty)} P(u) M(u)^{\sigma_c} < (1-\delta')^B P(Q) M(Q)^{\sigma_c}.
$$
Therefore, the solution $u$ scatters in $H^1$.
\end{proof}

\ 

\section{\bf SCATTERING: PROOF OF THEOREM}\label{sec:main-theorem}

\ 

\ 
In this section, we prove a Virial-Morawetz-type estimate to gain control over a suitable norm on large balls. The proof is concluded using the scattering criterion (Proposition \ref{scattering_criterion}). 

\ Consider the radial function
\begin{equation}
\label{eq:def_a}
a(x) = \begin{cases}
    |x|^2, \qquad |x|\leq \frac{R}{2} \\
    R|x|, \qquad |x| > R,
\end{cases}
\end{equation}
and which in the region $\frac{R}{2} \leq |x| \leq R$ verifies
\begin{equation}
\label{eq:props_a_annulus}
\partial_r a > 0, \quad \partial_r^2 a \geq 0, \quad |\partial^\alpha a| \lesssim_{\alpha} R |x|^{-\alpha+1} \quad \text{for } |\alpha| \geq 1.    
\end{equation}

We denote by $\partial_r$ the radial derivative, i.e. $\partial_r a = \nabla a \cdot \frac{x}{|x|}$. 
\begin{remark}
\begin{itemize}
\item  The angular derivative is denoted by $\hat \nabla$ and defined by
$\hat \nabla u = \nabla u - \frac{x}{|x|} \left( \frac{x}{|x|} \cdot \nabla u \right)$. It verifies that $$| \nabla u |^2 = |\partial_r u|^2 + |\hat \nabla u |^2.$$
In fact, $|\hat \nabla u|^2 = \left| \nabla u - \frac{x}{|x|} ( \frac{x}{|x|} \cdot \nabla u ) \right|^2 = |\nabla u|^2 + | \frac{x}{|x|} \cdot \nabla u|^2 + 2 Re \langle \nabla u, - \frac{x}{|x|} ( \frac{x}{|x|} \cdot \nabla u ) \rangle = |\nabla u |^2 + |\partial_r u|^2 - 2 |\frac{x}{|x|} \nabla u |^2 = |\nabla u |^2 - |\partial_r u|^2.$
\item For $|x|\leq \frac{R}{2}$, denoting $x=(x_1,x_2,...,x_N)$ and using sub-indexes to denote derivatives, we have
$$a_j := \partial_{x_j} a = 2 x_j, \quad \text{thus} \quad \nabla a = 2 x.$$
$$a_{jk} := \partial_{x_k}\partial_{x_j} a = 2 \delta_{jk}, \quad \Delta a = 2N, \qquad \Delta^2 a := \Delta (\Delta a) = 0.$$
\item  For $|x|>R$, 
$$a_j = R \frac{x_j}{|x|}, \quad \text{thus} \quad \nabla a = R \frac{x}{|x|}.$$
$$a_{jk} = \frac{R}{|x|} \left( \delta_{jk} - \frac{x_j}{|x|} \frac{x_k}{|x|} \right), \quad \Delta a = \frac{R}{|x|} (N-1), \qquad \Delta^2 a = 0.$$
\end{itemize}
   
\end{remark}


\begin{lemma} (Morawetz identity)
\label{lema:morawetz_identity}
    Let $a:\R^N \rightarrow \R$ be the smooth weight defined in \eqref{eq:def_a}. Define
\begin{equation}
    z(t) = \int a(x) |u(x,t)|^2 dx,    
\end{equation}
where $u$ is a solution of \eqref{GHP}.
Then
$$z'(t) = 2 Im \int (\nabla a \cdot \nabla u) \bar u dx,$$
\begin{equation}
\label{eq:morawetz_id}
\begin{split}
z''(t) = &\int -4\left(\frac{1}{2}-\frac{1}{p}\right) \left(I_{\gamma} * |v|^p \right) |u|^p \Delta a dx - \int \Delta^2 a |u|^2 dx + 4 \int Re(a_{jk} \bar u_j u_k) dx \\
& - \frac{4(N-\gamma)}{p} \int \int \nabla a(x) \cdot \frac{(x-y)}{|x-y|^{N-\gamma+2}} |u(y)|^p |u(x)|^p dy dx \\
& - 2 \int \nabla V \cdot \nabla a |u|^2 dx.
\end{split}
\end{equation}
\end{lemma}

\begin{remark}
    The second line of \eqref{eq:morawetz_id} can be written as
    $$
    \frac{2(N-\gamma)}{p} \int \int (\nabla a(x) - \nabla a(y)) \cdot \frac{(x-y)}{|x-y|^{N-\gamma+2}} |u(y)|^p |u(x)|^p dy dx.
    $$
    In particular if $a(x) = |x|^2$, it becomes 
    $$ \frac{2(N-\gamma)}{p} \int \int 2(x-y) \cdot \frac{(x-y)}{|x-y|^{N-\gamma+2}} |u(y)|^p dy |u(x)|^p dx  =  
    \frac{4(N-\gamma)}{p} \int (I_\gamma \ast |u|^p) |u(x)|^p dx.$$
\end{remark}

\begin{proposition} (Morawetz estimate)
\label{prop:morawez_estimate}
Let $T>0$ and let $V : \R^3 \to \R$ be a radially symmetric potential satisfying \eqref{poten} (i), $V\geq 0$, $x \cdot \nabla V \leq 0$ and $\partial_r V \in L^q$, for $\frac{3}{2} \leq q \leq \infty$. Let $u$ be a $H^1$-solution of \eqref{GHP}  satisfying the hypothesis \eqref{eq:only_one_hypo}. For $R=R(\delta, M(u),Q) > 0$ sufficiently large we have

\begin{equation}
\label{eq:morawetz_estimate}
\frac{1}{T} \int_0^T P(\chi_R u(t)) dt \lesssim_{\delta} \begin{cases}
    \frac{R}{T} + \frac{1}{R^{\frac{(N-1) B}{N}}} + o_R(1) \quad &\text{if} \quad (N-1)B < 2N \\
    \frac{R}{T} + \frac{1}{R^2} + o_R(1) \quad &\text{if} \quad (N-1)B \geq 2N.
\end{cases}    
\end{equation}
\end{proposition}

\begin{proof} Consider $\mathcal{M}(t)=z'(t)$.
 By Cauchy-Schwarz,  
 \begin{equation}\label{eq:sup_M}
 \sup_{t \in \R} |\mathcal{M}(t)| \leq 2 \int |u| |\nabla u| |\nabla a| dx\lesssim \|\nabla a\|_{L^\infty}\|u\|_2\|\nabla u\|_2\lesssim R.
 \end{equation}
 
From Lemma \ref{lema:morawetz_identity},
\begin{equation}
\begin{split}
\mathcal{M}'(t) =& \int |u|^2 (-\Delta^2 a) + 4 Re (a_{jk} \bar u_j u_k) dx \\
&- 4 \int \Delta a \left( \frac{1}{2}-\frac{1}{p}\right) \left( I_\gamma \ast|u|^p \right) |u|^p \\
&- \frac{4(N-\gamma)}{p} \int \int \nabla a \cdot \frac{x-y}{|x-y|^{N-\gamma+2}} |v(y)|^p dy |v(x)|^p dx \\
& - 2 \int \nabla V \cdot \nabla a |u|^2 dx.
\end{split}
\end{equation}
Except for the last term, the estimates are the same as in \cite{ARORA-GH-19}. One gets
\begin{equation}
\label{eq:Mor_est_1}
\begin{split}
\mathcal{M}'(t) \geq & 8 \int_{|x| \leq \frac{R}{2}} |\chi_R \nabla u|^2 - \frac{4B}{p} \int_{|x|\leq \frac{R}{2}} \left( I_\gamma \ast |\chi_R u|^p \right) |\chi_R u|^p \\
& + C_1 \iint_\Omega \left[\left(1-\frac{1}{2}\frac{R}{|x|} \right)x - \left(1-\frac{1}{2}\frac{R}{|y|} \right)y \right] \frac{x-y}{|x-y|^{N-\gamma+2}} |u(x)|^p dy |v(y)|^p dxdy\\
&- C_2 \int_{|x|>\frac{R}{2}} \left( I_\gamma \ast |u|^p \right) |u|^p dx - \frac{1}{R^2} M[u] - 2 \int \nabla V \cdot \nabla a |u|^2 dx,
\end{split}
\end{equation}
where $C_1, C_2 >0$ are constants and $\Omega$ is the region
$$
\Omega = \left\{ (x,y) \in \R^N \times \R^N : |x| > \frac{R}{2} \} \cup \{ (x,y) \in \R^N \times \R^N : |y| > \frac{R}{2} \right\}.
$$

Following \cite{ARORA-GH-19} (see also \cite{Saanouni-23-hartree}), we estimate the integral on $\Omega$ subdividing it into three regions. We denote $A:= \left[\left(1-\frac{1}{2}\frac{R}{|x|} \right)x - \left(1-\frac{1}{2}\frac{R}{|y|} \right)y \right]$.
\begin{itemize}
    \item Region I: we consider $x>\frac{R}{2}$ and $y>\frac{R}{2}$. We observe that
     $A \lesssim |x-y|$, then
     $$
     \iint_{\begin{array}{c}
          |x|>R/2\\
          |y|>R/2 
     \end{array}} |A|\frac{|x-y|}{|x-y|^{N-\gamma+2}}|u(y)|^p |u(x)|^p dxdy \lesssim
     \int_{|x|>R/2} \left( |x|^{-(N-\gamma)} * |u|^p\right) |u|^p dx.
     $$
     By Hölder inequality, Hardy inequality, and Radial Sobolev inequality (Proposition \ref{radial-sobolev-Lp}), we get
     \begin{align*}
     \int_{|x|>R/2} \left( |x|^{-(N-\gamma)} * |u|^p\right) |u|^p dx &\lesssim \||x|^{N-\gamma}*|u|^p\|_{L^{\frac{2N}{N-\gamma}}} \|u\|_{L^{\frac{2Np}{N+\gamma}}}^p \\
     &\lesssim \|u\|_{L^{\frac{2Np}{N+\gamma}}}^{2p} \\
     &\lesssim \frac{1}{R^{\frac{(N-1)B}{N}}}\|u\|_{L^2}^{\frac{B+2(N+\gamma)}{N}} \|\nabla u\|_{L^2}^{\frac{B}{N}} \\
     &\lesssim \frac{1}{R^{\frac{(N-1)B}{N}}}.
     \end{align*}
    \item Region II: we consider $x<\frac{R}{2}$ and $y>\frac{R}{2}$. Following \cite{ARORA-GH-19} we also have $A \lesssim |y| \sim |x-y|$, then
    $$
     \iint_{\begin{array}{c}
          |x|<R/2\\
          |y|>R/2 
     \end{array}} |A|\frac{|x-y|}{|x-y|^{N-\gamma+2}}|u(y)|^p |u(x)|^p dxdy \lesssim
     \int_{|y|>R/2} \int \frac{1}{|y-x|^{N-\gamma}} |u(x)|^p dx |u(y)|^p dy.
     $$
     Proceeding as in Region I, 
     \begin{align*}
     \int_{|y|>R/2} \left( |y|^{-(N-\gamma)} * |u|^p\right) |u|^p dy \lesssim \frac{1}{R^{\frac{(N-1)B}{N}}}.
     \end{align*}
     \item Region III: we consider symmetric region $x>\frac{R}{2}$ and $y<\frac{R}{2}$. The estimation is analogous:
     \begin{align*}
     \iint_{\begin{array}{c}
          |x|>R/2\\
          |y|<R/2 
     \end{array}} |A|\frac{|x-y|}{|x-y|^{N-\gamma+2}}|u(y)|^p |u(x)|^p dxdy \lesssim
     \int_{|x|>R/2} \left( |x|^{-(N-\gamma)} * |u|^p\right) |u|^p dy \lesssim \frac{1}{R^{\frac{(N-1)B}{N}}}.
     \end{align*}
     
\end{itemize}

Continuing the computation in \eqref{eq:Mor_est_1}, using Proposition \ref{prop:com_chi} we have
$$
\mathcal{M}'(t) \geq 8 \delta P(\chi_R u) + C_1 \frac{1}{R^{\frac{(N-1)B}{N}}}
- C_2 \int_{|x|>\frac{R}{2}} (I_\gamma * |u|^p) |u|^p dx - \frac{1}{R^2} M(u) - 2 \int \nabla V \cdot \nabla a |u|^2 dx.
$$
Proceeding as the integrals on $\Omega$ we also have
$$
\int_{|x|>R/2} \left( |x|^{-(N-\gamma)} * |u|^p\right) |u|^p dy \lesssim \frac{1}{R^{\frac{(N-1)B}{N}}},
$$
thus
\begin{equation}
\label{eq:mora_P_chi_u}
P(\chi_R u) \lesssim \mathcal{M}'(t) - C_1 \frac{1}{R^{\frac{(N-1)B}{N}}} + C_2 \frac{1}{R^{\frac{(N-1)B}{N}}} + \frac{1}{R^2} M(u) + 2 \int \nabla V \cdot \nabla a |u|^2 dx.    
\end{equation}

We estimate the term with the potential $V$ and show that it can be made arbitrarily small for $R$ big enough. By the definition of the function $a$ we have

$$ \int \nabla V \cdot \nabla a |u|^2 dx  = \int_{|x| < \frac{R}{2}} 2 x \cdot \nabla V |u|^2 dx + \int_{\frac{R}{2} \leq |x| \leq R} \nabla a \cdot \nabla V  |u|^2 dx + \int_{|x| > R} R \frac{x}{|x|} \cdot \nabla V  |u|^2 dx. $$

By the hypothesis $x \cdot \nabla V \leq 0$, we get 
$$ \int \nabla V \cdot \nabla a |u|^2 dx \leq  \int_{\frac{R}{2} \leq |x| \leq R} \nabla a \cdot \nabla V   |u|^2 dx.$$

Since $a$ and $V$ are radial, $\nabla a$ and $\nabla V$ are colinear and since $\partial_r a > 0$ and $\frac{x}{|x|}\cdot \nabla V = \partial_r V < 0$, we have $\nabla a \cdot \nabla V = -|\nabla a| |\nabla V| = -|\partial_r a| |\partial_r V|$. Using \eqref{eq:props_a_annulus} on $\frac{R}{2} \leq |x| \leq R$,  
$$ -|\partial_r a| |\partial_r V| \lesssim - R | \frac{x}{|x|} \cdot \nabla V| \lesssim - R \frac{1}{R} |x \cdot \nabla V|.$$
Thus
$$ \int \nabla V \cdot \nabla a |u|^2 dx \lesssim - \int_{\frac{R}{2} \leq |x| \leq R} |x \cdot \nabla V| |u|^2 dx.$$

On the other hand, the condition $x \cdot \nabla V  \in L^q$ for $\frac{3}{2} \leq q \leq \infty$, implies
$$ \int x \cdot \nabla V|u|^2 dx \leq \|x \cdot \nabla V\|_{L^q}\|u(t)\|_{L^{\frac{2q}{q-1}}}^2 \lesssim \|u\|_{H^1}^2 \lesssim 1,$$
where we used the Sobolev embedding $H^1(\R^3) \subset L^{\frac{2q}{q-1}}(\R^3)$ for any $\frac{3}{2} \leq q \leq \infty$. Thus,
$$ \int \nabla V \cdot \nabla a |u|^2 dx \leq o_R(1).
$$

Integrating \eqref{eq:mora_P_chi_u} on $[0,T]$ we obtain
$$
\int_0^T P(\chi_R u) dt \lesssim \sup_{t \in [0,T]} |\mathcal{M}(t)| + C_3 \frac{T}{R^{\frac{(N-1)B}{N}}} + \frac{T}{R^2} M(u) + T \ o_R(1). 
$$

Therefore, using \eqref{eq:sup_M},
$$
\frac{1}{T} \int_0^T P(\chi_R u(t)) dt \lesssim \frac{R}{T} + \frac{1}{R^{\frac{(N-1)B}{N}}} + \frac{1}{R^2} + o_R(1).
$$
If $(N-1)B < 2N$ then $\frac{1}{T} \int_0^T P(\chi_R u(t)) dt \lesssim \frac{R}{T} + \frac{1}{R^{\frac{(N-1)B}{N}}} + o_R(1)$, otherwise $\frac{1}{T} \int_0^T P(\chi_R u(t)) dt \lesssim \frac{R}{T} + \frac{1}{R^2} + o_R(1)$.
\end{proof}

We now show an energy evacuation result.
\begin{proposition}
\label{energy_evacuation}
There exist a sequence of times $t_n \to \infty$ and a sequence of radii $R_n \to \infty$ such that
$$
\lim_{n \to \infty} P(\chi_{R_n} u(x,t_n)) = 0.
$$
\end{proposition}
\begin{proof}
Taking $T=R^3$ in Proposition \ref{prop:morawez_estimate}, in any of the two cases in \eqref{eq:morawetz_estimate} we have for $R$ sufficiently large that
$$
\frac{1}{R^3} \int_0^{R^3} P(\chi_{R} u(t)) dt \lesssim \frac{1}{R^2} + \frac{1}{R^{\frac{(N-1)B}{N}}} + o_R(1).
$$
Since $u$ is global and $P(u(t)) \leq E(u(t)) = E(u_0)$ is bounded, by the mean value theorem there are sequences $t_n \to \infty$ and $R_n \to \infty$ such that
$$
\lim_{n \to \infty} P(\chi_{R_n} u(t_n))= 0.
$$
\end{proof}


\ We finally prove the main theorem.

{\bf Proof of Theorem \ref{theo:scattering}} The hypothesis \eqref{eq:only_one_hypo} implies that $P(u)$ is uniformly bounded in time. Since the energy $E(u)$ is preserved we have that $\|\nabla u\|^2 + \int V(x) |u|^2 dx = \|\Lambda u\|^2$ is bounded and then since $V>0$, the solution $u$ is global and uniformly bounded in $H^1$. 

Choose $\epsilon$ and $R$ as in the scattering criterion (Proposition \ref{scattering_criterion}) with $t_n \to\infty$ and $R_n \to \infty$ as in Proposition \ref{energy_evacuation}. Taking $n$ large so that $R_n \geq R$, using Hölder we have 
$$
\int_{|x|\leq R} |u(x,t_n)|^2 dx \lesssim R^{\frac{N(p-2)}{p}} \left(\int_{|x|<R_n} |u(x,t_n)|^{p}\right)^{2/p}.
$$
On the other hand, 
\begin{align*}
\left( \int_{|x|<R} |u(x,t_n)|^{p} \right)^2 &= \int_{|x|<R} \int_{|y|<R} \frac{|u(x,t_n)|^p |u(y,t_n)|^p}{|x-y|^{N-\gamma}} |x-y|^{N-\gamma} dx dy \\
& \leq (2R)^{N-\gamma} \int_{|x|<R_n} \int_{|y|<R_n} \frac{|u(x,t_n)|^p |u(y,t_n)|^p}{|x-y|^{N-\gamma}} dx dy \\
& \leq (2R)^{N-\gamma} \int_{|x|<R_n} \int_{|y|<R_n} \frac{1}{|x-y|^{N-\gamma}} |u(y,t_n)|^p dy |u(x,t_n)|^p dx \\
& = (2R)^{N-\gamma} P(\chi_{R_n} u(t_n)).    
\end{align*}
Proposition \ref{energy_evacuation} implies that
$$
\int_{|x|\leq R} |u(x,t_n)|^2 dx \to 0 \quad \text{as} \quad n \to \infty,
$$
and thus Proposition \ref{scattering_criterion} implies that $u$ scatters in $H^1$ forward in time.

\bibliographystyle{abbrv}
\bibliography{bib}	

\begin{thebibliography}{10}

\bibitem{ARORA-GH-19}
A.~K. Arora.
\newblock {Scattering of radial data in the focusing NLS and generalized Hartree equations}.
\newblock {\em Discrete and Continuous Dynamical Systems}, 39(11):6643--6668, 2019.

\bibitem{ARORA-ROUDENKO-22}
A.~K. Arora and S.~Roudenko.
\newblock {Global Behavior of Solutions to the Focusing Generalized Hartree Equation}.
\newblock {\em Michigan Mathematical Journal}, 71(3):619 -- 672, 2022.

\bibitem{CAZENAVEBOOK}
T.~Cazenave.
\newblock {\em Semilinear {S}chr\"odinger equations}, volume~10 of {\em Courant Lecture Notes in Mathematics}.
\newblock New York University, Courant Institute of Mathematical Sciences, New York; American Mathematical Society, Providence, RI, 2003.

\bibitem{Dod-Mur}
B.~Dodson and J.~Murphy.
\newblock A new proof of scattering below the ground state for the 3{D} radial focusing cubic {NLS}.
\newblock {\em Proc. Amer. Math. Soc.}, 145(11):4859--4867, 2017.

\bibitem{Gao_Wang-20-hartree}
Y.~Gao and Z.~Wang.
\newblock Below and beyond the mass–energy threshold: scattering for the hartree equation with radial data in $\pmb {d \ge 5}$.
\newblock {\em Zeitschrift für angewandte Mathematik und Physik}, 71(52), 2020.

\bibitem{GV}
J.~Ginibre and G.~Velo.
\newblock On a class of nonlinear {S}chr\"odinger equations. {I}. {T}he {C}auchy problem, general case.
\newblock {\em J. Funct. Anal.}, 32(1):1--32, 1979.

\bibitem{Hamano-Ikeda-20}
M.~Hamano and M.~Ikeda.
\newblock Global dynamics below the ground state for the focusing schrödinger equation with a potential.
\newblock {\em Journal of Evolution Equations}, 20(3):1131--1172, 2020.

\bibitem{HONG}
Y.~Hong.
\newblock Scattering for a nonlinear {S}chr\"{o}dinger equation with a potential.
\newblock {\em Commun. Pure Appl. Anal.}, 15(5):1571--1601, 2016.

\bibitem{KENIG}
C.~E. Kenig and F.~Merle.
\newblock Global well-posedness, scattering and blow-up for the energy-critical, focusing, non-linear {S}chr\"odinger equation in the radial case.
\newblock {\em Invent. Math.}, 166(3):645--675, 2006.

\bibitem{Lieb-Choquard-77}
E.~H. Lieb.
\newblock Existence and uniqueness of the minimizing solution of choquard's nonlinear equation.
\newblock {\em Studies in Applied Mathematics}, 57(2):93--105, 1977.

\bibitem{MOROZ-VSchaf-2013}
V.~Moroz and J.~{Van Schaftingen}.
\newblock Groundstates of nonlinear choquard equations: Existence, qualitative properties and decay asymptotics.
\newblock {\em Journal of Functional Analysis}, 265(2):153--184, 2013.

\bibitem{Saanouni-23-hartree}
T.~Saanouni.
\newblock Scattering theory for a class of radial focusing inhomogeneous hartree equations.
\newblock {\em Potential Analysis}, 58(4):617--643, 2023.

\bibitem{Tao}
T.~Tao.
\newblock On the asymptotic behavior of large radial data for a focusing non-linear {S}chr\"{o}dinger equation.
\newblock {\em Dyn. Partial Differ. Equ.}, 1(1):1--48, 2004.

\end{thebibliography}

\end{document}